\documentclass[11pt,a4]{article}

\usepackage[utf8]{inputenx}

\usepackage{appendix}
\usepackage{lmodern}
\usepackage{textcomp}
\usepackage[english]{babel}
\usepackage[T1]{fontenc}
\usepackage{bm}
\usepackage{amssymb,amsthm,amsmath}
\usepackage{indentfirst}
\usepackage{color}
\usepackage{graphicx}
\usepackage{enumerate}
\usepackage{caption}
\usepackage{empheq}
\usepackage{xcolor}
\usepackage{color}
\usepackage{amssymb}
\usepackage{booktabs}
\usepackage[font=small,labelfont=bf,labelsep=period,tableposition=top]{caption}
\usepackage{hyperref}

\newtheorem{theorem}{Theorem}[section]
\newtheorem{lemma}[theorem]{Lemma}
\newtheorem{proposition}[theorem]{Proposition}
\newtheorem{remark}[theorem]{Remark}
\newtheorem{definition}{Definition}[section]

\usepackage{ifthen}
\provideboolean{shownotes}\setboolean{shownotes}{true}
\newcommand{\margnote}[1]{\ifthenelse{\boolean{shownotes}}
	{\marginpar{\raggedright\tiny\texttt{#1}}}{}}

\hypersetup
{   bookmarks=false,
	colorlinks=true,
	urlcolor=blue,
	pdfauthor = {Andrea Aspri},
	pdftitle = {Stability draft},
	pdfcreator = {LaTeX, TeXmaker, pdfLaTeX, Hyperref}
}
	
\author{
	Andrea ASPRI\footnote{Johann Radon Insitute for Computational and Applied Mathematics (RICAM), Altenbergerstraße 69, 4040 Linz, Austria, \texttt{andrea.aspri@ricam.oeaw.ac.at}}\ ,  
	Elena BERETTA\footnote{Dipartimento di Matematica,
		Politecnico di Milano, Via Edoardo Bonardi 9 - 20133 Milano (Italy), \texttt{elena.beretta@polimi.it}}\ , 
	Edi ROSSET\footnote{Dipartimento di Matematica e Geoscienze, 	Universit\`a degli Studi di Trieste, via Valerio
			12/1 - 34127 Trieste (ITALY), \texttt{rossedi@units.it}}}
			
\title{On an elastic model arising from volcanology: an analysis of the direct and inverse problem}

\date{}

\begin{document}
\baselineskip=13pt
\maketitle
\begin{abstract}
In this paper we investigate a mathematical model arising from volcanology describing surface deformation effects generated by a magma chamber embedded into Earth’s interior and exerting on it a uniform hydrostatic pressure. The modeling assumptions translate mathematically into a Neumann boundary value problem for the classical Lam\'e system in a half-space with an embedded pressurized cavity. 
We establish well-posedness of the problem in suitable weighted Sobolev spaces and analyse the inverse problem of determining the pressurized cavity from partial measurements of the displacement field proving uniqueness and stability estimates.
\end{abstract}
\vskip.15cm

{\sf Keywords.} Lam\'e system; pressurized cavity; Neumann problem; half-space; weighted Sobolev spaces; well-posedness; inverse problem; uniqueness; stability estimates.
\vskip.15cm

{\sf 2010 AMS subject classifications.}
35R30  (35J57, 74B05, 86A60) \ \\

\medskip
\noindent
\emph{How to cite this paper:
	This paper has been accepted in
	J. Differential Equations, 265 (2018), 6400–-6423,
	and the final publication is available at\\
	https://doi.org/10.1016/j.jde.2018.07.031.}
\section{Introduction}
In this paper we investigate a linear elastic model describing surface deformations in a volcanic area induced by a magma chamber embedded in Earth's crust. 
From the mathematical point of view we introduce a simplified version of the model assuming the crust to be a half-space in a homogeneous and isotropic medium and the magma chamber {to be} a cavity subjected to a constant pressure on its boundary (for more details see, for example, \cite{Aspri-Beretta-Mascia,BattHill09,Lisowski06,Segall10}). 
More precisely, given $\mathbb{C}$ a fourth-order isotropic and homogeneous elastic tensor, with Lam\'e parameters $\lambda$ and $\mu$, denoting by $\bm{u}$ the displacement vector and by $\mathbb{R}^3_-$ the half-space, we end up with the following linear elastostatic boundary value problem
	\begin{equation}\label{direct problem intro}
		\begin{cases}
		\textrm{div}(\mathbb{C}\widehat{\nabla}\bm{u})=\bm{0} & \textrm{in}\, \mathbb{R}^3_-\setminus \overline{C}\\
		(\mathbb{C}\widehat{\nabla}\bm{u})\bm{n}=p\bm{n} & \textrm{on}\, \partial C\\
		(\mathbb{C}\widehat{\nabla}\bm{u})\bm{e}_3=\bm{0} & \textrm{on}\, \mathbb{R}^2\\
		\end{cases}
	\end{equation}
where $\widehat{\nabla}{\bm{u}}$ is the strain tensor, $C$ is the cavity, $p>0$ represents the pressure acting on the boundary of the cavity, $\bm{n}$ is the outer unit normal vector on $\partial C$ and $\bm{e}_3=(0,0,1)$.
 
The main purpose of this paper is to derive quantitative stability estimates for the inverse problem of identifying the pressurized cavity $C$ from one measurement of the displacement provided on a portion of the boundary of the half-space. 

In order to address this issue, we first analyse the well-posedness of \eqref{direct problem intro} under the assumption that $\partial C$ is Lipschitz.
We highlight that for the well-posedness we can either impose explicitly some decay conditions at infinity for $\bm{u}$ and $\nabla\bm{u}$ (see, for example, \cite{Aspri-Beretta-Mascia}) or, more suitably for our purposes, set the analysis in some weighted Sobolev spaces where the decay conditions are expressed by means of weights.
In particular, we will show the well-posedness in this weighted Sobolev space
	\begin{equation}\label{weighted sobolev space intro}
		{H^{1}_w(\mathbb{R}^3_-\setminus \overline{C})}=\Big\{\bm{u}\in \mathcal{D}'(\mathbb{R}^3_-\setminus \overline{C}), \frac{\bm{u}}{(1+|\bm{x}|^2)^{1/2}}\in L^2(\mathbb{R}^3_-\setminus \overline{C}), \nabla\bm{u}\in L^2(\mathbb{R}^3_-\setminus \overline{C})\Big\},
	\end{equation}
where $\mathcal{D}'(\mathbb{R}^3_-\setminus \overline{C})$ is the space of distributions in $\mathbb{R}^3_-\setminus \overline{C}$, with the norm given by
	\begin{equation}
		\|\bm{u}\|^2_{H^{1}_w(\mathbb{R}^3_-\setminus \overline{C})}=\biggl(\|(1+|\bm{x}|^2)^{-1/2}\bm{u}\|^2_{L^2(\mathbb{R}^3_-\setminus \overline{C})}+\|\nabla \bm{u}\|^2_{L^2(\mathbb{R}^3_-\setminus \overline{C})}\biggr).
	\end{equation} 	 
Even if the analysis of the well-posedness of general elastic problems in the half-space via weighted Sobolev spaces is known, see \cite{Amrouche-Dambrine-Raudin}, we would like to emphasize that in our framework we have two principal difficulties and novelties to treat: the first one concerns the fact that the problem is stated in an unbounded domain with unbounded boundary and with non homogeneous Neumann boundary conditions on the cavity. 
The second one is related to the derivation of quantitative stability estimates of the solution in $H^1_w(\mathbb{R}^3_-\setminus \overline{C})$. 
To this end, we need to derive a quantitative weighted Poincar\'e inequality and a Korn-type inequality in $\mathbb{R}^3_-\setminus \overline{C}$.
As far as we know, these inequalities are known in a quantitative way only for bounded domains, see \cite{Alessandrini-Morassi-Rosset}, and for conical domains, see \cite{Kondracev-Oleinik}. 

Therefore, the first part of this paper is devoted to prove quantitative Poincar\'e and Korn inequalities in $\mathbb{R}^3_-\setminus \overline{C}$ using some a priori information on the cavity $C$. 
This allows us to derive the following estimate for the unique solution $\bm{u}$ of problem \eqref{direct problem intro} 
	\begin{equation*}
		\|\bm{u}\|_{H^{1}_w(\mathbb{R}^3_-\setminus \overline{C})}\leq c p, 
	\end{equation*}
where the constant $c$ depends on the Lam\'e parameters, on the Lipschitz character of $\partial C$ and on the distance of $C$ from the boundary of the half-space. 
This estimate is fundamental for the direct problem and it is also necessary to establish stability estimates for the cavity in terms of the measurements. 

To prove the stability result for the inverse problem we need stronger regularity on the cavity. We follow and adapt when needed the results contained in \cite{Morassi-Rosset,Morassi-Rosset1}. 
From the point of view of the rate of convergence it is well known that, despite of smoothness assumptions on the cavity, only a weak rate of logaritmic type is expected. In our case we are able to prove a log-log type estimate and not the optimal logaritmic one proved for the scalar case (see \cite{Alessandrini-Beretta-Rosset-Vessella}) due to the lack of a doubling inequality at the boundary for the solutions of the Lam\'e system. 

The paper is organized as follows. 
In Section \ref{sec: notation and definitions} we give the notation used in the rest of the paper and definition on the regularity of the domains.
In Section \ref{sec: the direct problem - well-posedness and regularity result}, we set the analysis of the elastic problem in the weighted Sobolev space, proving first the constructive Poincar\'e and Korn inequalities and then giving the result of the well-posedness. Section \ref{sec: the inverse problem - uniqueness and stability estimate} is devoted to the analysis of the inverse problem and the derivation of the uniqueness and stability results.

\section{On Some Notation and Useful Definitions}\label{sec: notation and definitions}
In this section we set up notation and some definitions paying specific attention to the regularity of bounded domains. 

In the sequel, we denote scalar quantities in italic type, e.g. $\lambda, \mu, \nu$,
points and vectors in bold italic type, e.g.  $\bm{x}, \bm{y}, \bm{z}$ and $\bm{u}, \bm{v}, \bm{w}$, matrices and second-order tensors in bold type, e.g.  $\mathbf{A}, \mathbf{B}, \mathbf{C}$, and fourth-order tensors in blackboard bold type, e.g.  $\mathbb{A}, \mathbb{B}, \mathbb{C}$. 

The transpose of a second-order tensor $\mathbf{A}$ is denoted by $\mathbf{A}^T$
and its symmetric part by 
	\begin{equation*}
		\widehat{\mathbf{A}}=\tfrac{1}{2}\left(\mathbf{A}+\mathbf{A}^T\right).
	\end{equation*}
To indicate the inner product between two vectors $\bm{u}$ and $\bm{v}$ we use
$\bm{u}\cdot \bm{v}=\sum_{i} u_{i} v_{i}$
whereas for second-order tensors $\mathbf{A}:\mathbf{B}=\sum_{i,j}a_{ij} b_{ij}$. 
The tensor product of two vectors $\bm{u}$ and $\bm{v}$ is denoted by $\bm{u}\otimes \bm{v}=u_i v_j$. 
Similarly, $\mathbf{A}\otimes \mathbf{B}=A_{ij}B_{hk}$ represents the tensor product between matrices. 
With $|\mathbf{A}|$ we mean the norm induced by the inner product between second-order tensors, that is 
	\begin{equation*}
		|\mathbf{A}|=\sqrt{\mathbf{A}:\mathbf{A}}.
	\end{equation*}
We denote the open half-space 
	\begin{equation*}
		\{\bm{x}=(x_1,x_2,x_3)\in \mathbb{R}^3\,:\,x_3< 0\}=\mathbb{R}^3_-
	\end{equation*}
and we represent with $\mathbb{R}^2$ its boundary, that is the set $\{\bm{x}=(x_1,x_2,x_3)\in \mathbb{R}^3\,:\,x_3=0\}$. 
The set $B^-_r(\bm{0})$ denotes the half ball of centre $\bm{0}$ and radius $r$, that is 
	\begin{equation*}
		B^-_r(\bm{0})=\{\bm{x}\in\mathbb{R}^3\, :\, x^2_1+x^2_2+x^2_3<r^2, x_3<0\}. 
	\end{equation*}
With $B'_{r}(\bm{0})$ we mean the circle of centre $\bm{0}$ and radius $r$, namely 	
	\begin{equation*}
		B'_{r}(\bm{0})=\{\bm{x}\in\mathbb{R}^2\, :\, x^2_1+x^2_2<r^2\}.
	\end{equation*}
We denote with $d(A,B)$ the distance between the two sets $A$ and $B$, that is
	\begin{equation*}
		d(A,B):=\inf\{|\bm{x}-\bm{y}|\, :\, \bm{x}\in A, \bm{y}\in B\}
	\end{equation*}
and with $d_{\mathcal{H}}(A,B)$ their Hausdorff distance, namely
	\begin{equation*}
		d_{\mathcal{H}}(A,B):=\max\{\sup\limits_{\bm{x}\in A}\inf\limits_{\bm{y}\in B}|\bm{x}-\bm{y}|,\sup\limits_{\bm{y}\in B}\inf\limits_{\bm{x}\in A}|\bm{x}-\bm{y}|\}.
	\end{equation*}	 
The unit outer normal vector at the boundary of a regular domain is represented by $\bm{n}$.

\subsection{Domain Regularity}
In the following sections, the costants appearing in the inequalities will depend on some a priori information of the constitutive parameters of the linear elastic model and on the a priori geometric and regularity assumptions on the cavity. For this reason it is important to recall the definition of $C^{k,\alpha}$ regularity for a bounded domain.   
\begin{definition}[$C^{k,\alpha}$ regularity]\ \\
Let $\Omega$ be a bounded domain in $\mathbb{R}^3$. Given $k, \alpha$, with $k\in\mathbb{N}$ and $0<\alpha\leq 1$, we say that a portion $S$ of $\partial \Omega$ is of class $C^{k,\alpha}$ with constant $r_0$, $E_0$, if for any $\bm{P}\in S$, there exists a rigid transformation of coordinates under which we have $\bm{P}=\bm{0}$ and 
	\begin{equation*}
		\Omega\cap B_{r_0}(\bm{0})=\{\bm{x}\in B_{r_0}(\bm{0})\, :\, x_3>\psi(\bm{x}')\},
	\end{equation*}
where ${\psi}$ is a $C^{k,\alpha}$ function on $B'_{r_0}(\bm{0})\subset \mathbb{R}^2$ such that
	\begin{equation*}
		\begin{aligned}
			{\psi}(\bm{0})&=0,\\
			\nabla{\psi}(\bm{0})&=\mathbf{0}, \qquad \text{for}\, k\geq 1\\
			\|{\psi}\|_{C^{k,\alpha}(B_{r_0}(\bm{0}))}&\leq E_0.
		\end{aligned}
	\end{equation*}
When $k=0, \alpha=1$, we also say that $S$ is of Lipschitz class with constants $r_0$, $E_0$.
\end{definition}

\section{The Direct Problem}\label{sec: the direct problem - well-posedness and regularity result}
In this section, we will analyse the well-posedness of the following linear elastostatic boundary value problem 
	\begin{equation}\label{direct pb}
		\begin{cases}
		\textrm{div}(\mathbb{C}\widehat{\nabla}\bm{u})=\bm{0} & \textrm{in}\, \mathbb{R}^3_-\setminus \overline{C}\\
		(\mathbb{C}\widehat{\nabla}\bm{u})\bm{n}=p\bm{n} & \textrm{on}\, \partial C\\
		(\mathbb{C}\widehat{\nabla}\bm{u})\bm{e}_3=\bm{0} & \textrm{on}\, \mathbb{R}^2\\
		\end{cases}
	\end{equation}
where $p>0$ represents the pressure, $C$ is the cavity,  $\bm{n}$ is the outer unit normal vector on $\partial C$ and $\bm{e}_3=(0,0,1)$.  $\mathbb{C}$ is the fourth-order isotropic and homogeneous elastic tensor given by 
	\begin{equation*}
		\mathbb{C}:=\lambda \mathbf{I}\otimes \mathbf{I}+2\mu \mathbb{I},
	\end{equation*}		
where $\lambda$ and $\mu$ are the two constants Lam\'e parameters, $\mathbf{I}$ is the identity matrix in $\mathbb{R}^3$ and $\mathbb{I}$ is the fourth-order identity tensor such that $\mathbb{I}\mathbf{A}=\widehat{\mathbf{A}}$, for any second-order tensor $\mathbf{A}$.

We first provide some physical information on the cavity $C$ and the elastic tensor $\mathbb{C}$.  
\subsection{Main assumptions and a priori information}\label{sec: some a priori information}
For the study of the direct problem \eqref{direct pb}, we assume that the constant Lam\'e parameters satisfy the inequalities 
	\begin{equation} \label{bounds_lame}
		3\lambda+2\mu>0\quad \textrm{and}\quad \mu>0,
	\end{equation}
that is the tensor $\mathbb{C}$ is strongly convex: 
	\begin{equation}\label{pos_def}
		\mathbb{C}\widehat{\mathbf{A}}: \widehat{\mathbf{A}} \geq \xi_0 |\widehat{\mathbf{A}}|^2,
	\end{equation}
with $\xi_0 = \min\{2\mu, 2\mu+3\lambda\}$, see \cite{Gurtin}.	

{The cavity $C$ is supposed to be a bounded domain with Lipschitz regularity}, that is
	\begin{align}
		\partial C\, \textrm{is Lipschitz with constants}\, r_0\, \textrm{and}\, E_0.\label{regularity of C}
	\end{align}
Additionally, we impose some a priori information on the size of the cavity and its distance from the boundary of the half-space. 
In particular, denoting with $\text{diam}(A)$ the diameter of a set $A$, we require
	\begin{align}
		 B^-_{2D_0}(\bm{0})&\supset C\label{a priori information on C},\\
		d(C, \mathbb{R}^2)&\geq D_0 \label{first a priori information},\\
		\text{diam}(C)&< D_0,
		\label{second a priori information}
	\end{align}
where, without loss of generality, we can assume that the constant $D_0>1$.

\begin{remark}
From here on, for simplicity of reading, we omit the dependence of some constants on the Lam\'e coefficients $\lambda$ and $\mu$, on the parameters $r_0, E_0, D_0$ related to the a priori information on the cavity $C$ and on $s_0$ {which represents the radius of the circle where the measurements are collected. For its definition see Section 4.}
\end{remark}

We highlight that, since we are in the half-space, namely an unbounded domain with unbounded boundary, the study of the well-posedness of the direct problem \eqref{direct pb} can be done either imposing some decay conditions at infinity for the function $\bm{u}$ and $\nabla \bm{u}$ (see for example \cite{Aspri-Beretta-Mascia}) or setting the analysis in a suitable weighted Sobolev space.
We choose this second strategy. 
So we recall the definition of the weighted Sobolev space for domains of type $\mathbb{R}^3_-\setminus \overline{C}$. To do that, we denote the space of the indefinitely differentiable functions with compact support in ${\mathbb{R}}^3_-\setminus \overline{C}$ by $\mathcal{D}({\mathbb{R}}^3_-\setminus \overline{C})$ and with $\mathcal{D}'(\mathbb{R}^3_-\setminus \overline{C})$ its dual space, that is the space of distributions.
\begin{definition}[Weighted Sobolev space]
Given the function
	\begin{equation}\label{weight}
		\rho=(1+|\bm{x}|^2)^{1/2},
	\end{equation}
we define 
	\begin{equation}\label{w. Sobolev space}
		{H^{1}_w(\mathbb{R}^3_-\setminus \overline{C})}=\Big\{\bm{u}\in \mathcal{D}'(\mathbb{R}^3_-\setminus \overline{C}), \frac{\bm{u}}{\rho}\in L^2(\mathbb{R}^3_-\setminus \overline{C}), \nabla\bm{u}\in L^2(\mathbb{R}^3_-\setminus \overline{C})\Big\}.
	\end{equation}
\end{definition}
This weighted Sobolev space is a reflexive Banach space, see for example \cite{Amrouche-Bonzom} and references therein, equipped with its natural norm
	\begin{equation}\label{def weighted sobolev space}
		\|\bm{u}\|^2_{H^{1}_w(\mathbb{R}^3_-\setminus \overline{C})}=\biggl(\|\rho^{-1}\bm{u}\|^2_{L^2(\mathbb{R}^3_-\setminus \overline{C})}+\|\nabla \bm{u}\|^2_{L^2(\mathbb{R}^3_-\setminus \overline{C})}\biggr).
	\end{equation} 	
We recall that the weight is chosen so that the space $\mathcal{D}(\overline{\mathbb{R}}^3_-\setminus {C})$ is dense in $H^{1}_w(\mathbb{R}^3_-\setminus \overline{C})$, see \cite{Amrouche-Bonzom,Hanouzet}. 
When we deal with bounded domains $D\subset(\mathbb{R}^3_-\setminus \overline{C})$ we emphasize that $H^1_w(D)$ reduces to $H^1(D)$ regularity, hence the usual trace theorems hold.
For generalizations and more details on weighted Sobolev spaces see, for example, \cite{Amrouche-Bonzom,Amrouche-Dambrine-Raudin,Hanouzet} and references therein.

The study of the well-posedness of \eqref{direct pb} is therefore accomplished using Lax-Milgram theorem in $H^1_w(\mathbb{R}^3_-\setminus \overline{C})$ space.
We stress that the use of the weighted Sobolev space is the natural approach to obtain a quantitative $H^1_w(\mathbb{R}^3_-\setminus \overline{C})$ estimate in terms of the boundary data. To this end we need first to have constructive Poincar\'e and Korn-type inequalities. 

\subsection{Weighted Poincar\'e Inequality and Korn-type Inequality} 
In this section we want to prove a weighted Poincar\'e inequality and a Korn-type inequality in $\mathbb{R}^3_-\setminus \overline{C}$ using a suitable partition of unity.
  
Let us consider two {half balls} $B^-_r(\bm{0})$ and $B^-_R(\bm{0})$, with $r<R$, such that 
	\begin{equation*}
		C\subset B^-_r(\bm{0})\subset B^-_R(\bm{0}).
	\end{equation*}
Using the a priori information \eqref{first a priori information} and \eqref{second a priori information} on the cavity $C$, we fix 
	\begin{equation*}
		r=3 D_0\qquad \textrm{and}\qquad R=4 D_0.  
	\end{equation*}
We consider a specific partition of unity of $\mathbb{R}^3_-$. In particular, we take $\varphi_1, \varphi_2 \in C^{\infty}(\mathbb{R}^3_-)$ such that 
	\begin{equation}\label{property of varphi}
		0\leq \varphi_1,\varphi_2\leq 1\quad \textrm{and}\quad \varphi_1+\varphi_2=1\quad \textrm{in}\, \mathbb{R}^3_-,
	\end{equation}
with
	\begin{align}
			\varphi_2&=0, \quad \varphi_1=1,\quad &\textrm{in}\, {B^-_r(\bm{0})},\  \label{supp phi2}\\ 
			\varphi_1&=0, \quad \varphi_2=1,\quad &\textrm{in}\, \{|\bm{x}|\geq R\}\cap {{\mathbb{R}}}^3_-, \ \label{supp phi1}\\
			|\nabla\varphi_1|&\leq \frac{c}{\rho},\quad |\nabla\varphi_2|\leq \frac{c}{\rho},\quad &\textrm{in}\, B^-_R(\bm{0})\setminus B^-_r(\bm{0}),\label{behaviour grad phi} 
	\end{align}
where $c$ is an absolute positive constant thanks to the choice made for $r$ and $R$. 

In the following it is useful to split $\partial B^-_R(\bm{0})=\partial B^h_R(\bm{0})\cup \partial B^b_R(\bm{0})$, where $\partial B^h_R(\bm{0})$ is the circle $B'_R(\bm{0})$ whereas $\partial B^b_R(\bm{0})$ is the spherical cap. We first prove Poincar\'e inequality
\begin{theorem}[Weighted Poincar\'e Inequality]\label{weighted poincaré inequality}
For any function $\bm{u}\in H^1_w(\mathbb{R}^3_-\setminus \overline{C})$ there exists a positive constant $c$, with $c=c(r_0,E_0,D_0)$, such that
	\begin{equation}\label{WPI}
		\int\limits_{\mathbb{R}^3_-\setminus \overline{C}}\bigg|\frac{\bm{u}}{\rho}\bigg|^2\, d\bm{x}\leq c \int\limits_{\mathbb{R}^3_-\setminus \overline{C}}|\nabla\bm{u}|^2\, d\bm{x},
	\end{equation}
where $\rho$ is defined in \eqref{weight}.	
\end{theorem}
\begin{proof}
From \eqref{property of varphi} we find that
	\begin{equation*}
		\bigg\| \frac{\bm{u}}{\rho}\bigg\|^2_{L^2(\mathbb{R}^3_-\setminus \overline{C})}\leq 2\bigg(\bigg\| \varphi_1\frac{\bm{u}}{\rho}\bigg\|^2_{L^2(\mathbb{R}^3_-\setminus \overline{C})}+\bigg\|\varphi_2 \frac{\bm{u}}{\rho}\bigg\|^2_{L^2(\mathbb{R}^3_-\setminus \overline{C})}\bigg):=2(\mathcal{N}_1+\mathcal{N}_2).
	\end{equation*} 
We study $\mathcal{N}_1$ and $\mathcal{N}_2$. \\
From the property \eqref{supp phi1} and since
$\rho^{-1}\leq 1$, we get
	\begin{equation}\label{mathcal N1}
		\mathcal{N}_1=\bigg\| \varphi_1\frac{\bm{u}}{\rho}\bigg\|^2_{L^2(B^-_R(\bm{0})\setminus \overline{C})}\leq \|\varphi_1{\bm{u}}\|^2_{L^2(B^-_R(\bm{0})\setminus \overline{C})}.
	\end{equation}
Therefore, since $\varphi_1=0$ on $\partial B^b_R(\bm{0})$, we use the {quantitative} Poincar\'e inequality for functions vanishing on a portion of the boundary of a bounded domain, see for instance \cite{Alessandrini-Morassi-Rosset} (Theorem 3.3, in particular Example 3.6), finding
	\begin{equation}\label{equ phi1*u}
		\|\varphi_1 \bm{u}\|^2_{L^2(B^-_R(\bm{0})\setminus \overline{C})}\leq c\, \|\nabla(\varphi_1 \bm{u})\|^2_{L^2(B^-_R(\bm{0})\setminus \overline{C})},
	\end{equation}
where $c$ is a positive constant such that $c=c(r_0,E_0,D_0)$. 
In this way, we obtain
	\begin{equation}\label{grad(phi1u)}
		\|\nabla(\varphi_1 \bm{u})\|^2_{L^2(B^-_R(\bm{0})\setminus \overline{C})}\leq 2 \left(\|\bm{u}\otimes \nabla \varphi_1\|^2_{L^2(B^-_R(\bm{0})\setminus \overline{C})}+\|\varphi_1\nabla \bm{u}\|^2_{L^2(B^-_R(\bm{0})\setminus \overline{C})}\right).
	\end{equation}
Now, from the property \eqref{behaviour grad phi}, we have
	\begin{equation}\label{u otimes grad phi_1}
		\|\bm{u}\otimes \nabla \varphi_1\|^2_{L^2(B^-_R(\bm{0})\setminus \overline{C})}=\int\limits_{B^-_R(\bm{0})\setminus \overline{C}}|\bm{u}|^2|\nabla\varphi_1|^2\, d\bm{x}\leq c \displaystyle\int\limits_{B^-_R(\bm{0})\setminus \overline{B^-_r(\bm{0})}}\frac{|\bm{u}|^2}{\rho^2}\, d\bm{x} .
	\end{equation}
By \eqref{equ phi1*u}, \eqref{grad(phi1u)}, \eqref{u otimes grad phi_1} and recalling 
\eqref{property of varphi}, we have
	\begin{equation}\label{phi1*u}
		\|\varphi_1 \bm{u}\|^2_{L^2(B^-_R(\bm{0})\setminus \overline{C})}\leq c \left(\bigg\|\frac{\bm{u}}{\rho}\bigg\|^2_{L^2(\{|\bm{x}|>r\}\, \cap\, \mathbb{R}^3_-)}+\|\nabla \bm{u}\|^2_{L^2(B^-_R(\bm{0})\setminus \overline{C})}\right).
	\end{equation}
Applying Hardy's inequality for the exterior of a {half ball} in the half-space (see \cite{Kondracev-Oleinik}, Lemma 3, p. 83) to the first term in the right-hand side of the inequality \eqref{phi1*u}, we find
	\begin{equation}\label{hardy inequality}
		\bigg\|\frac{\bm{u}}{\rho}\bigg\|^2_{L^2(\{|\bm{x}|>r\}\, \cap\,  \mathbb{R}^3_-)}\leq c\, \|\nabla{\bm{u}}\|^2_{L^2(\{|\bm{x}|>r\}\, \cap\, \mathbb{R}^3_-)}.
	\end{equation}
Inserting \eqref{hardy inequality} in \eqref{phi1*u} and then going back to \eqref{mathcal N1}, we have
	\begin{equation}\label{inequ N_1}
		\begin{aligned}
		\mathcal{N}_1=\bigg\| \varphi_1\frac{\bm{u}}{\rho}\bigg\|^2_{L^2(B^-_R(\bm{0})\setminus \overline{C})}&\leq c\left(\|\nabla \bm{u}\|^2_{L^2(B^-_R(\bm{0})\setminus \overline{C})}+\|\nabla{\bm{u}}\|^2_{L^2(\{|\bm{x}|>r\}\,\cap\, \mathbb{R}^3_-)}\right)\\
		&\leq c\,\|\nabla \bm{u}\|^2_{L^2(\mathbb{R}^3_-\setminus \overline{C})}.
		\end{aligned}
	\end{equation}
Analogously, using the properties \eqref{property of varphi} and \eqref{supp phi2} and applying again Hardy's inequality (see \cite{Kondracev-Oleinik}, Lemma 3, p. 83), we find 
	\begin{equation}\label{inequ N_2}
		\begin{aligned}
 		\mathcal{N}_2=\bigg\|\varphi_2\frac{\bm{u}}{\rho}\bigg\|^2_{L^2(\mathbb{R}^3_-\setminus \overline{C})}\leq \bigg\|\frac{\bm{u}}{\rho}\bigg\|^2_{L^2(\{|\bm{x}|>r\}\, \cap\, \mathbb{R}^3_-)}&\leq c \|\nabla{\bm{u}}\|^2_{L^2(\{|\bm{x}|>r\}\, \cap\, \mathbb{R}^3_-)}\\
 		&\leq c\,\|\nabla \bm{u}\|^2_{L^2(\mathbb{R}^3_-\setminus \overline{C})}.
 		\end{aligned} 
	\end{equation}
Putting together the inequalities \eqref{inequ N_1} and \eqref{inequ N_2} we have the assertion.
\end{proof}
Before proving a Korn-type inequality in the exterior domain of a half-space, we state, for the reader's convenience, a slight modification of a lemma proved by Kondrat'ev and Oleinik in \cite{Kondracev-Oleinik} (see Lemma 5. p. 85) for the case of a function $\bm{u}\in H^1_w(\mathbb{R}^3_-\setminus \overline{C})$. This lemma will be useful in the proof of the Korn inequality.
\begin{lemma}\label{Lemma 5. Oleinik}
Let $\bm{u}\in H^1_w(\mathbb{R}^3_-\setminus \overline{C})$. For every $r'<r$ there exists a positive constant $c$ such that
	\begin{equation*}
		\|\nabla \bm{u}\|_{L^2(\{|\bm{x}|>r\}\, \cap\, \mathbb{R}^3_-)}\leq c \|\widehat{\nabla} \bm{u}\|_{L^2(\{|\bm{x}|>r'\}\, \cap\, \mathbb{R}^3_-)},
	\end{equation*}
where $c=c(r,r')$.	 
\end{lemma}  
Now, we are ready to prove the following quantitative Korn inequality.
\begin{theorem}[Korn-type Inequality]\label{korn inequality}
For any function $\bm{u}\in H^1_w(\mathbb{R}^3_-\setminus \overline{C})$ there exists a positive constant $c$, with $c=c(r_0,E_0,D_0)$, such that
	\begin{equation}\label{KI}
		\int\limits_{\mathbb{R}^3_-\setminus \overline{C}}|\nabla\bm{u}|^2\, d\bm{x}\leq c \int\limits_{\mathbb{R}^3_-\setminus \overline{C}}|\widehat{\nabla}\bm{u}|^2\, d\bm{x}.
	\end{equation}
\end{theorem}
\begin{proof}
From the definition of the function $\varphi_1, \varphi_2$, see \eqref{property of varphi}, we have
	\begin{equation*}
		\|\nabla \bm{u}\|^2_{L^2(\mathbb{R}^3_-\setminus \overline{C})} \leq 2\bigg(\|\nabla (\varphi_1\bm{u})\|^2_{L^2(\mathbb{R}^3_-\setminus \overline{C})} + \|\nabla(\varphi_2\bm{u}) \|^2_{L^2(\mathbb{R}^3_-\setminus \overline{C})}\bigg):=2(\mathcal{N}'_1+\mathcal{N}'_2).  
	\end{equation*}
We study, separately, the two terms $\mathcal{N}'_1$ and $\mathcal{N}'_2$.

By \eqref{supp phi1} we find
	\begin{equation}\label{mathcal N'_1}
		\begin{aligned}
			\mathcal{N}'_1=\|\nabla (\varphi_1\bm{u})\|^2_{L^2(\mathbb{R}^3_-\setminus \overline{C})}&=\|\nabla (\varphi_1\bm{u})\|^2_{L^2(B^-_R(\bm{0})\setminus \overline{C})},\\
		\end{aligned}
	\end{equation}  
hence, since $\varphi_1=0$ on $\partial B^b_R(\bm{0})$, we apply the quantitative Korn inequality for functions vanishing on a portion of the boundary of a bounded domain, see for instance \cite{Alessandrini-Morassi-Rosset} (Theorem 5.7), getting for $c=c(r_0,E_0,D_0)$	 
	\begin{equation*}
		\begin{aligned}
			\|\nabla (\varphi_1\bm{u})\|^2_{L^2(B^-_R(\bm{0})\setminus \overline{C})}&\leq c\, \|\widehat{\nabla} (\varphi_1\bm{u})\|^2_{L^2(B^-_R(\bm{0})\setminus \overline{C})}\\
			&\leq c \bigg(\|\bm{u}\otimes \nabla\varphi_1\|^2_{L^2(B^-_R(\bm{0})\setminus \overline{C})}+\|\varphi_1\widehat{\nabla}\bm{u}\|^2_{L^2(B^-_R(\bm{0})\setminus \overline{C})}\bigg),\\
		\end{aligned}
	\end{equation*}
where in the right side of the previous inequality we have used
	\begin{equation*}
		\|\widehat{\bm{u}\otimes \nabla}\varphi_1\|^2_{L^2(B^-_R(\bm{0})\setminus \overline{C})}\leq \|\bm{u}\otimes \nabla\varphi_1\|^2_{L^2(B^-_R(\bm{0})\setminus \overline{C})}.
	\end{equation*}		
From \eqref{u otimes grad phi_1}, the properties \eqref{property of varphi} and Hardy's inequality \eqref{hardy inequality}, we have
	\begin{equation*}
		\begin{aligned}
			\|\bm{u}\otimes \nabla\varphi_1\|^2_{L^2(B^-_R(\bm{0})\setminus \overline{C})}+&\|\varphi_1\widehat{\nabla}\bm{u}\|^2_{L^2(B^-_R(\bm{0})\setminus \overline{C})}\\
			&\hspace{2cm}\leq c \left(\bigg\|\frac{\bm{u}}{\rho}\bigg\|^2_{L^2(\{|\bm{x}|>r\}\, \cap\, \mathbb{R}^3_-)}+\|\widehat{\nabla}\bm{u}\|^2_{L^2(B^-_R(\bm{0})\setminus \overline{C})}\right)\\
			&\hspace{2cm}\leq c \left(\|\nabla\bm{u}\|^2_{L^2(\{|\bm{x}|>r\}\, \cap\, \mathbb{R}^3_-)}+\|\widehat{\nabla}\bm{u}\|^2_{L^2(B^-_R(\bm{0})\setminus \overline{C})}\right).
		\end{aligned}
	\end{equation*} 
Applying Lemma \ref{Lemma 5. Oleinik} to the first term in the right side of the previous formula, we find
	\begin{equation*}
		\|{\nabla \bm{u}}\|^2_{L^2(\{|\bm{x}|>r\}\,\cap\, \mathbb{R}^3_- )}\leq c \|{\widehat{\nabla} \bm{u}}\|^2_{L^2(\{|\bm{x}|>r'\}\,\cap\, \mathbb{R}^3_- )},
	\end{equation*}
where we choose $r'=2D_0< r$ so that $C\subset B^-_{r'}(\bm{0})$ (see the a priori information \eqref{a priori information on C}). Putting together all these results and going back to \eqref{mathcal N'_1}, we have
	\begin{equation}\label{N'_1}
		\begin{aligned}
			\mathcal{N}'_1\leq c \bigg(\|\widehat{\nabla}\bm{u}\|^2_{L^2(B^-_R(\bm{0})\setminus \overline{C})}+\|{\widehat{\nabla} \bm{u}}\|^2_{L^2(\{|\bm{x}|>r'\}\,\cap\, \mathbb{R}^3_- )}\bigg)\leq c\, \|\widehat{\nabla}\bm{u}\|^2_{L^2(\mathbb{R}^3_-\setminus \overline{C})}.
		\end{aligned}	
	\end{equation}
In a similar way, using the properties \eqref{property of varphi} and \eqref{supp phi2}, we find 
	\begin{equation*}
		\begin{aligned}
			\mathcal{N}'_2=\|\nabla(\varphi_2\bm{u})\|^2_{L^2(\mathbb{R}^3_-\setminus \overline{C})}&=\|\nabla(\varphi_2\bm{u})\|^2_{L^2(\{|\bm{x}|>r\}\,\cap\, \mathbb{R}^3_- )}\\
			&\leq c \bigg( \|\bm{u}\otimes \nabla\varphi_2\|^2_{L^2(\{|\bm{x}|>r\}\,\cap\, \mathbb{R}^3_- )}+\|\varphi_2\nabla\bm{u}\|^2_{L^2(\{|\bm{x}|>r\}\,\cap\, \mathbb{R}^3_- )} \bigg).
		\end{aligned}
	\end{equation*}
From the properties \eqref{property of varphi} and \eqref{behaviour grad phi} of $\varphi_2$, we get
	\begin{equation*}
		\begin{aligned}
			\|\bm{u}\otimes \nabla\varphi_2\|^2_{L^2(\{|\bm{x}|>r\}\,\cap\, \mathbb{R}^3_- )}+&\|\varphi_2\nabla\bm{u}\|^2_{L^2(\{|\bm{x}|>r\}\,\cap\, \mathbb{R}^3_- )}\\
			&\hspace{1.2cm}\leq c \left(\bigg\|\frac{\bm{u}}{\rho}\bigg\|^2_{L^2(\{|\bm{x}|>r\}\, \cap\, \mathbb{R}^3_-)}+\|\nabla\bm{u}\|^2_{L^2(\{|\bm{x}|>r\}\, \cap\, \mathbb{R}^3_-)}\right).
		\end{aligned}
	\end{equation*}
Using again Hardy's inequality \eqref{hardy inequality} and the result in Lemma \ref{Lemma 5. Oleinik} in the last two terms of the previous formula, we find
	\begin{equation}\label{N'_2}
		\mathcal{N}'_2\leq c\, \|\widehat{\nabla}\bm{u}\|^2_{L^2(\{|\bm{x}|>r'\}\, \cap\, \mathbb{R}^3_-)}\leq c\, \|\widehat{\nabla}\bm{u}\|^2_{L^2(\mathbb{R}^3_-\setminus \overline{C})} .
	\end{equation}
Finally, collecting the results in \eqref{N'_1} and \eqref{N'_2}, 
we have the assertion. 
\end{proof}
\subsection{Well-Posedness}
To study the well-posedness of problem \eqref{direct pb} we use a variational approach. 
We suppose, for the moment, $\bm{u}$ regular and the test functions $\bm{v}$ in $\mathcal{D}(\overline{\mathbb{R}}^3_-\setminus {C})$. Multiplying the equations in \eqref{direct pb} for the functions $\bm{v}$ and integrating in $\mathbb{R}^3_-\setminus \overline{C}$, we obtain
	\begin{equation*}
		\int\limits_{\mathbb{R}^3_-\setminus \overline{C}}\mathbb{C}\widehat{\nabla}\bm{u}:\widehat{\nabla}{\bm{v}}\, d\bm{x}=-p\int\limits_{\partial C}{\bm{n}}\cdot {\bm{v}}\, d\sigma(\bm{x}),\quad \forall {\bm{v}}\in \mathcal{D}(\overline{\mathbb{R}}^3_-\setminus {C}).
	\end{equation*}
Now, from the density property of the functional space $\mathcal{D}(\overline{\mathbb{R}}^3_-\setminus {C})$ into the weighted Sobolev space defined in \eqref{w. Sobolev space}, problem \eqref{direct pb} becomes:\\ 
\textit{
find $\bm{u}\in H^{1}_w(\mathbb{R}^3_-\setminus \overline{C})$ such that 
	\begin{equation}\label{variational formulation}
		a(\bm{u},\bm{v})=f(\bm{v}), \quad \forall {\bm{v}}\in H^{1}_w(\mathbb{R}^3_-\setminus \overline{C}),
	\end{equation}
where $a: H^1_w(\mathbb{R}^3_-\setminus \overline{C})\times H^1_w(\mathbb{R}^3_-\setminus \overline{C}) \to \mathbb{R}$ is the bilinear form given by
	\begin{equation}\label{bilinear form}
	 a(\bm{u},\bm{v})=\int\limits_{\mathbb{R}^3_-\setminus \overline{C}}\mathbb{C}\widehat{\nabla}\bm{u}:\widehat{\nabla}{\bm{v}}\, d\bm{x},
	\end{equation}
and $f: H^1_w(\mathbb{R}^3_-\setminus \overline{C})\to \mathbb{R}$ is the linear functional given by
 	\begin{equation}\label{linear functional f}
 		f(\bm{v})=-p\int\limits_{\partial C}{\bm{n}}\cdot {\bm{v}}\, d\sigma(\bm{x}).
 	\end{equation}
} 
Now, we can prove
\begin{theorem}
Problem \eqref{direct pb} admits a unique solution $\bm{u}\in H^{1}_w
(\mathbb{R}^3_-\setminus \overline{C})$ satisfying
	\begin{equation}\label{estimate of u in H1w}
		\|\bm{u}\|_{H^{1}_w(\mathbb{R}^3_-\setminus \overline{C})}\leq c p, 
	\end{equation}
where the constant $c=c(\lambda,\mu,r_0,E_0,D_0)$.	
\end{theorem}
\begin{proof}
To prove the well-posedness of problem \eqref{direct pb} we apply Lax-Milgram theorem to \eqref{variational formulation}.
Therefore, we need to prove the coercivity and the continuity property of the bilinear form \eqref{bilinear form} and the boundedness of the linear functional \eqref{linear functional f}.\\
\textit{Continuity of \eqref{bilinear form}}.\\
From the Cauchy-Schwarz inequality we have
	\begin{equation*}
		\begin{aligned}
			|a(\bm{u},\bm{v})|=\Bigg|\int\limits_{\mathbb{R}^3_-\setminus \overline{C}}\mathbb{C}\widehat{\nabla}\bm{u}:\widehat{\nabla}{\bm{v}}\, d\bm{x}\Bigg|& \leq c\, \|\widehat{\nabla}\bm{u}\|_{L^{2}(\mathbb{R}^3_-\setminus \overline{C})} \|\widehat{\nabla}\bm{v}\|_{L^{2}(\mathbb{R}^3_-\setminus \overline{C})}\\
			&\leq c\, \|\bm{u}\|_{H^{1}_w(\mathbb{R}^3_-\setminus \overline{C})} \|\bm{v}\|_{H^{1}_w(\mathbb{R}^3_-\setminus \overline{C})},
		\end{aligned}
	\end{equation*} 
where $c=c(\lambda,\mu)$.\\		
\textit{Coercivity of \eqref{bilinear form}}.\\
We apply the constructive Poincar\'e and Korn inequalities proved in Theorem \ref{weighted poincaré inequality}, Theorem \ref{korn inequality} and the strong convexity condition of $\mathbb{C}$, see \eqref{pos_def}. In detail, we have
	\begin{equation*}
		\begin{aligned}
			a(\bm{u},\bm{u})=\int\limits_{\mathbb{R}^3_-\setminus \overline{C}}\mathbb{C}\widehat{\nabla}\bm{u}:\widehat{\nabla}{\bm{u}}\, d\bm{x}&\geq c\|\widehat{\nabla}\bm{u}\|^2_{L^{2}(\mathbb{R}^3_-\setminus \overline{C})}\\
			&\geq c \|{\nabla}\bm{u}\|^2_{L^{2}(\mathbb{R}^3_-\setminus \overline{C})}\geq c \|\bm{u}\|^2_{H^{1}_w(\mathbb{R}^3_-\setminus \overline{C})},
		\end{aligned}
	\end{equation*}
where the constant $c=c(\lambda,\mu,r_0,E_0,D_0)$. \\	
\textit{Boundedness of \eqref{linear functional f}}.\\
Let us take $B^-_{2D_0}(\bm{0})$. Then applying the trace theorem for bounded domains, we find
	\begin{equation*}
		\begin{aligned}
 		\Bigg|-p\int\limits_{\partial C}{\bm{n}}\cdot {\bm{v}}\, d\sigma(\bm{x})\Bigg|&\leq c\,p\|\bm{v}\|_{L^2(\partial C)}\leq c\, p\Bigg(\bigg\|\frac{\bm{v}}{\rho}\bigg\|_{L^2((B^-_{2D_0}(\bm{0}))\setminus \overline{C})}+{\|\nabla \bm{v}\|_{L^2((B^-_{2D_0}(\bm{0}))\setminus \overline{C})}}\Bigg)\\ 
 		&\leq c\, p \|\bm{v}\|_{H^1_w(\mathbb{R}^3_-\setminus \overline{C})}.
 		\end{aligned} 
 	\end{equation*}
Applying the Lax-Milgram theorem we obtain the well-posedness of problem \eqref{direct pb}. Moreover, by means of the strong convexity condition of $\mathbb{C}$, see \eqref{pos_def}, and from the application of the Korn and Poincar\'e inequalities, see Theorem \ref{weighted poincaré inequality} and Theorem \ref{korn inequality}, we find that 
	\begin{equation*}
		\|\bm{u}\|^2_{H^{1}_w(\mathbb{R}^3_-\setminus C)}\leq \Bigg|\int\limits_{\mathbb{R}^3_-\setminus \overline{C}}\mathbb{C}\widehat{\nabla}\bm{u}:\widehat{\nabla}{\bm{u}}\, d\bm{x}\Bigg|\leq\Bigg| p\int\limits_{\partial C}\bm{n}\cdot \bm{u}\, d\sigma(\bm{x})\Bigg|\leq c p \|\bm{u}\|_{H^{1}_w(\mathbb{R}^3_-\setminus \overline{C})}, 
	\end{equation*}	
where the constant $c=c(\lambda,\mu,r_0,E_0,D_0)$, hence the assertion of the theorem follows.	
\end{proof}
\section{The Inverse Problem: Uniqueness and Stability Estimate}\label{sec: the inverse problem - uniqueness and stability estimate}
In this section we will investigate the following 
inverse problem: \textit{given the displacement vector $\bm{u}$ on a portion of the boundary of the half-space can we detect uniquely and in a stable way the cavity $C$?}

We suppose to have the measurements on $B'_{s_0}(\bm{0})=\{\bm{x}\in\mathbb{R}^2\, :\, x_1^2+x^2_2< s^2_0\}$ contained in  $\{x_3=0\}$, with $s_0<D_0$.
To prove a stability estimate for the inverse problem we need to require more regularity on $C$ than \eqref{regularity of C}. In particular, we suppose that:  
	\begin{equation}\label{more regularity for C}
		\partial C\,  \textrm{is of class}\, C^3\, \textrm{with constant}\, r_0\, \textrm{and}\, E_0.		
	\end{equation}
In addition, we recall that $C$ satisfy the a priori information \eqref{a priori information on C}, \eqref{first a priori information} and \eqref{second a priori information}.
We also assume that
	\begin{equation*}
		\mathbb{R}^3_-\setminus \overline{C}\, \textrm{is connected}.
	\end{equation*} 	
Before proceeding, we highlight that the proof of the uniqueness and the stability result is based on the possibility to build the displacement field 
	\begin{equation}\label{definition of ubar}
		\overline{\bm{u}}=\frac{p}{3\lambda+2\mu}\bm{x}
	\end{equation}
so to reduce problem \eqref{direct pb} to a problem with homogeneous Neumann boundary conditions on the boundary of the cavity. 
A straightforward calculation shows that $\overline{\bm{u}}$ satisfies the Lam\'e system and the boundary condition on $C$ satisfied by $\bm{u}$.

In this way, the function
	\begin{equation}\label{definition of the function w}
 		\bm{w}:=\bm{u}-\overline{\bm{u}}, 
 	\end{equation}	
satisfies the following boundary value problem	 
	\begin{equation}\label{pb for w}
		\begin{cases}
		\textrm{div}(\mathbb{C}\widehat{\nabla}\bm{w})=\bm{0} & \textrm{in}\, \mathbb{R}^3_-\setminus \overline{C}\\
		(\mathbb{C}\widehat{\nabla}\bm{w})\bm{n}=0 & \textrm{on}\, \partial C\\
		(\mathbb{C}\widehat{\nabla}\bm{w})\bm{e}_3=-p\bm{e}_3 & \textrm{on}\, \mathbb{R}^2\\
    	\bm{w} + \bm{\overline u}\in H^1_w(\mathbb{R}^3_-\setminus \overline{C}), 
		\end{cases}
	\end{equation}
where $\bm{e}_3=(0,0,1)$. The inverse problem reduces therefore to determine the cavity $C$ from a single pair of Cauchy data on $B'_{s_0}(\bm{0})$ of the solution to problem \eqref{pb for w}.

In the sequel, we denote $\bm{w}_i=\bm{u}_i-\overline{\bm{u}}$, for $i=1,2$, where $\bm{w}_i$ and $\bm{u}_i$ are respectively the solutions to \eqref{pb for w} and \eqref{direct pb} with $C=C_i$, for $i=1,2$. It immediately follows that
	\begin{equation}\label{w_1-w_2=u_1-u_2}
		\bm{w}_1-\bm{w}_2=\bm{u}_1-\bm{u}_2,\, \qquad \textrm{in}\,\, {\mathbb{R}}^3_-\setminus \overline{(C_1 \cup C_2)}.
	\end{equation}
Moreover, we indicate with 
	\begin{equation}\label{G unbounded component}
		G\ \, \textrm{the unbounded connected component of}\ \, \mathbb{R}^3_-\setminus(\overline{C_1\cup C_2}).
	\end{equation}	
Notice that $B'_{s_0}(\bm{0})\subset \partial G$.	
\subsection{Uniqueness}
Although the procedure to get the uniqueness result for the inverse problem is known in the literature, for the reader's convenience, we give a sketch of the proof. 
\begin{theorem}\label{th of uniqueness}
Given the single pair of Cauchy data $\{\bm{w},-p\bm{e}_3\}$ on $B'_{s_0}(\bm{0})$ there exists at most one pair ($C$,$\bm{w}$) satisfying problem \eqref{pb for w}.
\end{theorem}
\begin{remark}
Theorem \ref{th of uniqueness} is true under weaker regularity assumptions on the cavity $C$. In fact, it is sufficient to have $\partial C$ of Lipschitz class.
\end{remark} 
\begin{proof}[Proof of Theorem \ref{th of uniqueness}]
Suppose by contradiction that there exist two cavities $C_1$ and $C_2$, with $C_1\neq C_2$, and the corrisponding vector displacements $\bm{w}_1$, $\bm{w}_2$ such that 
	\begin{equation*}
		\bm{w}_1\Big|_{B'_{s_0}(\bm{0})}=\bm{w}_2\Big|_{B'_{s_0}(\bm{0})}=\bm{w},\qquad (\mathbb{C}\widehat{\nabla}\bm{w}_1)\bm{e}_3\Big|_{B'_{s_0}(\bm{0})}=(\mathbb{C}\widehat{\nabla}\bm{w}_2)\bm{e}_3\Big|_{B'_{s_0}(\bm{0})}=-p\bm{e}_3.
	\end{equation*}
From the unique continuation theorem for solution to the Lam\'e system, see \cite{weck}, we have
	\begin{equation*}
		\bm{w}_1=\bm{w}_2,\,\qquad \textrm{in}\,\ G, 
	\end{equation*} 
where $G$ is defined in \eqref{G unbounded component}.	
Next, we consider two different cases of intersection of the domains $C_1$ and $C_2$. 
In fact, all the other possibilities can be reduced to these two configurations. 
For example, we take the domain $D$ as in Figure \ref{fig. unicita},
\begin{figure}[h]
\centering
\includegraphics[scale=0.4]{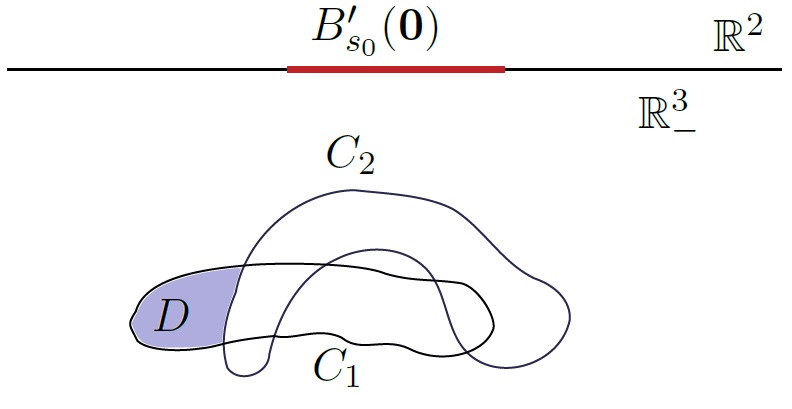}
\caption{The two domains $C_1$ and $C_2$ and the domain $D$.}\label{fig. unicita}
\end{figure}
where the function $\bm{w_2}$ is well-defined and satisfies the elastostatic equations, finding
	\begin{equation*}
		\int\limits_{D}\bigl(\mathbb{C}\widehat{\nabla}\bm{w}_2\bigr):\widehat{\nabla}\bm{w}_2\, d\bm{x}=\int\limits_{\partial D}\bigl(\mathbb{C}\widehat{\nabla}\bm{w}_2\bm{n}\bigr)\cdot\bm{w}_2\, d\sigma(\bm{x}).
	\end{equation*}
Since $\mathbb{C}\widehat{\nabla}\bm{w}_2\bm{n}=0$ on $\partial D\cap \partial C_2$ and by the unique continuation property $\mathbb{C}\widehat{\nabla}\bm{w}_2\bm{n}=\mathbb{C}\widehat{\nabla}\bm{w}_1\bm{n}$ on $\Gamma_1=(\partial D\cap \partial C_1)\subset \partial G$, we get
	\begin{equation*}
		\int\limits_{D}\bigl(\mathbb{C}\widehat{\nabla}\bm{w}_2\bigr):\widehat{\nabla}\bm{w}_2\, d\bm{x}=\int\limits_{\Gamma_1}\bigl(\mathbb{C}\widehat{\nabla}\bm{w}_1\bm{n}\bigr)\cdot\bm{w}_2\, d\sigma(\bm{x})=0,
	\end{equation*}
hence $\bm{w}_2=\mathbf{A}\bm{x}+\bm{a}$ in $D$, where $\mathbf{A}\in\mathbb{R}^{3\times 3}$ is a skew-symmetric matrix and $\bm{a}\in\mathbb{R}^3$. From the unique continuation principle applied to $\bm{w}_2-\mathbf{A}\bm{x}-\bm{a}$ we obtain that $\bm{w}_2=\mathbf{A}\bm{x}+\bm{a}$ in $\mathbb{R}^3_-\setminus \overline{C}_2$, hence $\bigl(\mathbb{C}\widehat{\nabla}\bm{w}_2\bigr) \bm{n}=0$ on $B'_{s_0}(\bm{0})$, that is a contradiction.

The second case we analyse is related to the setting in Figure \ref{fig. unicita1}. To prove the contradiction, in this case we can consider, for example, the domain $D=D_1\cup D_2$, where $D_1$ and $D_2$ are represented in Figure \ref{fig. unicita1}. 
\begin{figure}[h]
\centering
\includegraphics[scale=0.4]{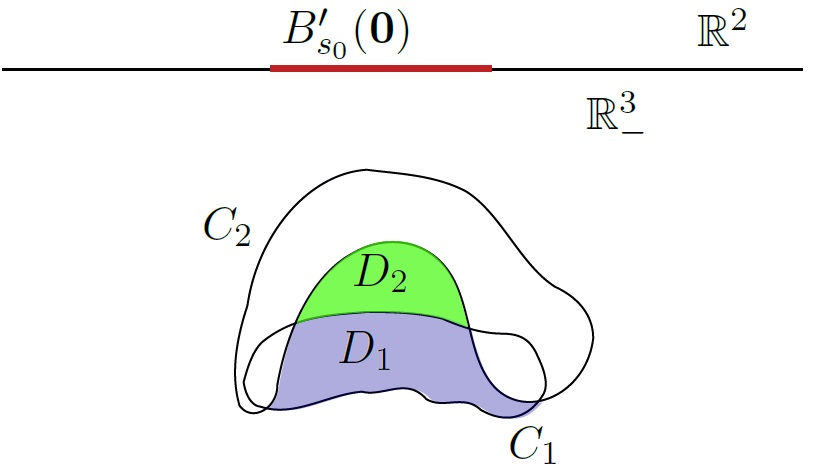}
\caption{The two domains $C_1$ and $C_2$ and the domain $D=D_1\cup D_2$.}\label{fig. unicita1}
\end{figure}
We emphasize that in this setting if we only consider the domain $D_1$ we will not be able to find a contradiction. In fact, from the unique continuation and the Green's formula we would not be able to prove that the energy of the system related to the function $\bm{w}_2$ is equal to zero in $D_1$. Instead, considering the region $D$ and taking  $\bm{w}_2$, which satisfies $\textrm{div}(\mathbb{C}\widehat{\nabla}\bm{w_2})=\bm{0}$ in $D$, from the Green's formula we find, as before,
	\begin{equation*}
		\int\limits_{D}\bigl(\mathbb{C}\widehat{\nabla}\bm{w}_2\bigr):\widehat{\nabla}\bm{w}_2\, d\bm{x}=0,
	\end{equation*} 
hence $\bm{w}_2=\mathbf{A}\bm{x}+\bm{a}$ in $D$, with $\mathbf{A}\in\mathbb{R}^{3\times 3}$ a skew-symmetric matrix and $\bm{a}\in\mathbb{R}^3$. Applying the unique continuation principle to $\bm{w}_2-\mathbf{A}\bm{x}-\bm{a}$ we have that $\bm{w}_2=\mathbf{A}\bm{x}+\bm{a}$ in $\mathbb{R}^3_-\setminus \overline{C}_2$, hence $\bigl(\mathbb{C}\widehat{\nabla}\bm{w}_2\bigr) \bm{n}=0$ on $B'_{s_0}(\bm{0})$, that is a contradiction.
\end{proof}

\subsection{Stability Estimate}
In this section we state and prove the stability estimates for our inverse problem by adapting the arguments contained in \cite{Morassi-Rosset} and \cite{Morassi-Rosset1} to the case of a pressurized cavity in an unbounded domain. 
In order to keep the proof of the main result as readable as possible and since the strategy to get the stability theorem is similar of the one obtained in \cite{Morassi-Rosset} we will not repeat all the details of the proofs of the auxiliary results we need. 
The main idea behind the stability result (Theorem \ref{th: stability estimate}) is to give a quantitative version of the uniqueness argument. More precisely, we combine two steps: 
\begin{enumerate}[(a)]
\item the propagation of the smallness of the Cauchy data up to the boundary of the cavities, leading to an integral estimate of the solutions (see Propositions \ref{prop: SECCD} and \ref{improved stability estimates of continuation from cauchy data}); 
\item an estimate of continuation from the interior (Proposition \ref{LPS}). 
\end{enumerate}
The basic tool for both steps is the three spheres inequality stated in Lemma \ref{3spheres inequality}.

In the sequel, for any $\varrho>0$, we denote by
\begin{equation*}
		\Omega_{\varrho}=\{\bm{x}\in \Omega\, :\, \textrm{dist}(\bm{x},\partial \Omega)> \varrho\}.
	\end{equation*}
We first prove a regularity result on the solution of problem \eqref{direct pb}.
To this end, we consider a bounded domain $Q\subset \mathbb{R}^3_-$ such that
	\begin{align}
		&\partial Q\in C^3\, \textrm{with constants}\, r_0, E_0, 	\label{Omega domain regularity} \\	
		&\overline{B^-_{\alpha D_0}(\bm{0})}\subset\subset \overline{Q}\subset\subset \overline{B^-_{\beta D_0}(\bm{0})},	\label{Omega domain inclusions}
	\end{align}
where $\alpha> 2$ and $\beta\geq 3$, with $\alpha <\beta$.		
Now, we have the following regularity estimate.
\begin{proposition}\label{th regularity estimate}
Under the assumptions \eqref{more regularity for C} for $C$ and \eqref{Omega domain regularity}, \eqref{Omega domain inclusions} for $Q$, the solution of problem \eqref{direct pb}, satisfies
	\begin{equation}\label{eq regularity estimate}
		\|\bm{u}\|_{C^{1,1/2}(\overline{Q}\setminus C)}\leq c p,
	\end{equation}
where the constant $c=c(\lambda,\mu,\alpha,\beta,r_0,E_0,D_0)$.
\end{proposition}
Before proving this theorem, we briefly recall the integral representation formula for the solution to problem \eqref{direct pb} derived in \cite{Aspri-Beretta-Mascia}. In particular, we define first the Neumann function, solution to
	\begin{equation}\label{Neumann function}
		\begin{cases}
			\textrm{div}(\mathbb{C} \widehat{\nabla}{\mathbf{N}}(\cdot,\bm{y}))={\delta_{\bm{y}}}\mathbf{I}
			& \textrm{in }\, \mathbb{R}^3_-\\
			(\mathbb{C}\widehat{\nabla}{\mathbf{N}}(\cdot,\bm{y}))\bm{n}=0 &  \textrm{on }\, \mathbb{R}^2\\
			{\mathbf{N}}=O(|\bm{x}|^{-1}),\qquad |\nabla{\mathbf{N}}|=O(|\bm{x}|^{-2}) & |\bm{x}|\to \infty,
		\end{cases}	
	\end{equation}
where ${\delta_{\bm{y}}}$ is the delta function centred at $\bm{y}\in\mathbb{R}^3_-$ and $\mathbf{I}$ is the identity matrix (see Theorem \ref{thm:fundsol} in Appendix for its explicit expression). 
Then, for any $\bm{y}\in  \overline{\mathbb{R}}^3_-\setminus \overline{C}$ we have
	\begin{equation}\label{representation formula}
		\bm{u}(\bm{y})=p\int\limits_{\partial C}\mathbf{N}(\bm{x},\bm{y})\bm{n}(\bm{x})\, d\sigma(\bm{x})-\int\limits_{\partial C}\big[\big(\mathbb{C}\widehat{\nabla}_{\bm{x}}\mathbf{N}(\bm{x},\bm{y})\big)\bm{n}(\bm{x})\big]^T\bm{f}(\bm{x})\, d\sigma(\bm{x}),
	\end{equation}
where $\bm{f}$ is the trace of $\bm{u}$ on $\partial C$ and $\bm{n}$ is the outer unit normal vector on $\partial C$ (for more details see \cite{Aspri-Beretta-Mascia}).

\begin{proof}[Proof of Theorem \ref{th regularity estimate}]
Since the kernels of the integral operators in \eqref{representation formula} are regular for $\bm{y}\in\partial Q\setminus\partial C$, we can estimate $D^k\bm{u}(\bm{y})$, for $k=0,1,2,3$, getting easily
	\begin{equation*}
		|D^k\bm{u}(\bm{y})|\leq p |\partial C|\, \sup\limits_{\mathclap{\substack{\bm{x}\in\partial C \vspace{0.05cm}\\ \bm{y}\in\partial Q\setminus\partial C}}}|D^k_{\bm{y}}\mathbf{N}(\bm{x},\bm{y})|+|\partial C|^{1/2}\|\bm{f}\|_{L^2(\partial C)}\,\sup\limits_{\mathclap{\substack{\bm{x}\in\partial C \vspace{0.05cm} \\ \bm{y}\in\partial Q\setminus\partial C}}}|D^k_{\bm{y}}(\mathbb{C}\widehat{\nabla}_{\bm{x}}\mathbf{N}(\bm{x},\bm{y}))|.
	\end{equation*}
From the regularity properties of $\mathbf{N}$ and Theorem \ref{thm:fundsol} in Appendix, we have
	\begin{equation}\label{Linf estimates of representation formula}
		\begin{aligned}
			\sup\limits_{\mathclap{\substack{\bm{x}\in\partial C \vspace{0.05cm}\\ \bm{y}\in\partial Q\setminus\partial C}}}|D^k_{\bm{y}}\mathbf{N}(\bm{x},\bm{y})|&\leq \frac{c}{D^{k+1}_0},\qquad\qquad
			\sup\limits_{\mathclap{\substack{\bm{x}\in\partial C \vspace{0.05cm} \\ \bm{y}\in\partial Q\setminus\partial C}}}|D^k_{\bm{y}}(\mathbb{C}\widehat{\nabla}_{\bm{x}}\mathbf{N}(\bm{x},\bm{y}))|&\leq \frac{c}{D^{k+2}_0},
		\end{aligned}
	\end{equation} 
where the constant $c=c(\lambda,\mu,\alpha)$. 
From the trace estimate applied in $Q$, we have
	\begin{equation}\label{estimate of f}
		\|\bm{f}\|_{L^2(\partial C)}\leq c\|\bm{u}\|_{H^1_w(\mathbb{R}^3_-\setminus \overline{C})}\leq c p,
	\end{equation}
hence, from \eqref{Linf estimates of representation formula} and \eqref{estimate of f}, we get
	\begin{equation*}
		|D^k\bm{u}(\bm{y})|\leq c p,
	\end{equation*}
where the constant $c=c(\lambda,\mu,\alpha,r_0,E_0,D_0)$. 
Therefore
	\begin{equation}\label{estimate on partialOmega setminus partial C}
		\|D^k\bm{u}\|_{L^{\infty}(\partial Q\setminus \partial C)}\leq c p,\qquad k=0,1,2,3.
	\end{equation}
Next, we apply the following global regularity estimate for the elastostatic system with Neumann boundary conditions (see \cite{Valent}, Theorem 6.6, p. 79) for $\bm{u}$ in $Q\setminus \overline{C}$, that is
	\begin{equation}\label{estimate H3}
		\|\bm{u}\|_{H^3(Q\setminus \overline{C})}\leq c\bigg(\|\bm{u}\|_{L^2(Q\setminus \overline{C})}+\|(\mathbb{C}\widehat{\nabla}\bm{u})\bm{n}\|_{W^{3/2,2}(\partial(Q\setminus \overline{C}))}\bigg),
	\end{equation}
with $c=c(\lambda,\mu,\alpha,\beta,r_0,E_0,D_0)$.	
Now, observe that by \eqref{estimate of u in H1w}
	\begin{equation}\label{estimate on L2(omega setminus c)}
		\|\bm{u}\|_{L^2(Q\setminus \overline{C})}\leq c\, \bigg\|\frac{\bm{u}}{\rho}\bigg\|_{L^2(\mathbb{R}^3_-\setminus \overline{C})}\leq c p, 
	\end{equation}
where $c=c(\lambda,\mu,\alpha,\beta,r_0,E_0,D_0)$, while, for the term $\|(\mathbb{C}\widehat{\nabla}\bm{u})\bm{n}\|_{W^{3/2,2}(\partial(Q\setminus \overline{C}))}$ we have
	\begin{equation*}
		(\mathbb{C}\widehat{\nabla}\bm{u})\bm{n}=\bm{0},\quad \textrm{on}\, 	\partial Q \cap \{x_3=0\},\qquad\qquad (\mathbb{C}\widehat{\nabla}\bm{u})\bm{n}=p\bm{n},\quad \textrm{on}\, \partial C,
	\end{equation*}
and since $\partial C$ is of class $C^3$ (see \eqref{more regularity for C}) it follows $(\mathbb{C}\widehat{\nabla}\bm{u})\bm{n}\in C^2(\partial C)$, hence
	\begin{equation}\label{estimate traction on partial c}
		\|(\mathbb{C}\widehat{\nabla}\bm{u})\bm{n}\|_{W^{3/2,2}(\partial C)}\leq c p.
	\end{equation}
Analogously, from the estimate \eqref{estimate on partialOmega setminus partial C} and the regularity of the boundary of $Q$, we find
	\begin{equation}\label{estimate traction on partial omega}
		\|(\mathbb{C}\widehat{\nabla}\bm{u})\bm{n}\|_{W^{3/2,2}(\partial Q\setminus \partial C)}\leq c p.
	\end{equation}
Therefore, collecting \eqref{estimate on L2(omega setminus c)}, \eqref{estimate traction on partial c} and \eqref{estimate traction on partial omega}, the estimate \eqref{estimate H3} gives
	\begin{equation*}
		\|\bm{u}\|_{H^3(Q\setminus \overline{C})}\leq c p.
	\end{equation*} 
Finally, applying the Sobolev embedding theorem \eqref{eq regularity estimate} follows.
\end{proof}

\begin{proposition}[Lipschitz propagation of smallness]\label{LPS}
Under the assumptions \eqref{bounds_lame}, \eqref{a priori information on C}, \eqref{first a priori information}, \eqref{second a priori information} and \eqref{more regularity for C}, let $\bm{w}$ be the solution to \eqref{pb for w}.
There exist $R\geq 3D_0$, $R=R(\lambda, \mu, r_0, E_0, D_0)$, and $s>1$, $s=s(\lambda, \mu, E_0)$, such that for every $\varrho>0$ and every $\overline{\bm{x}}\in (B^-_R(\bm{0})\setminus \overline{C})_{s\varrho}$, we have	
	\begin{equation}
		\int\limits_{B_{\varrho}(\overline{\bm{x}})}|\widehat{\nabla} \bm{w}|^2\, d\bm{x}\geq
				\frac{c}{e^{a \varrho^{-b}}}
		\int\limits_{B^-_R(\bm{0})\setminus \overline{C}}|\widehat{\nabla}\bm{w}|^2\, d\bm{x},
	\end{equation}
where $a,b,c>0$ depend on $\lambda,\mu,r_0,E_0$ and $D_0$.
\end{proposition}
The proof of this proposition is based on the application of the three-spheres inequality. For the reader's convenience we recall here only the statement of the theorem in the case of the linear elasticity in a homogenous and isotropic medium; for a more general case and its proof one can refer to \cite{Alessandrini-Morassi,Alessandrini-Morassi-Rosset-3SFERE}.

\begin{lemma}[Three spheres inequality]\label{3spheres inequality}
Let $\Omega$ be a bounded domain in $\mathbb{R}^3$. Let $\bm{w}\in
H^1(\Omega)$ be a solution to the Lam\'{e} system. There exists ${\vartheta^*}$,
$0<{\vartheta^*}\leq 1$, only depending on $\lambda$ and $\mu$ such that for every $r_1, r_2, r_3, \overline{r}$,
$0<r_1<r_2<r_3\leq {\vartheta^*}\overline{r}$, and for every $\bm{x}\in
\Omega_{\overline{r}}$ we have
	\begin{equation}\label{eq:3sphres}
   		\int\limits_{B_{r_{2}}(\bm{x})}|\widehat{\nabla}\bm{w}|^{2} \leq c
   		\left(\displaystyle\int\limits_{B_{r_{1}}(\bm{x})}|\widehat{\nabla}\bm{w}|^{2}
   		\right)^{\delta}\left(\displaystyle\int\limits_{B_{r_{3}}(\bm{x})}|\widehat{\nabla} \bm{w}|^{2}
   		\right)
   		^{1-\delta},
	\end{equation}
where $c>0$ and $\delta$, $0<\delta<1$, only depend on
$\lambda$, $\mu$, $\frac{r_{2}}{r_{3}}$ and are monotone  increasing functions of $\frac{r_{1}}{r_{3}}$.
\end{lemma}
Now, the Lipschitz propagation of smallness inequality can be proved. 
\begin{proof}[Proof of Proposition \ref{LPS}]
Let us denote $\Omega=B^-_R(\bm{0})\setminus \overline{C}$, with $R\geq 3D_0$ to be chosen later. Since the hemisphere $B^-_1(\bm{0})$ has Lipschitz boundary with absolute constants $r^*, E^*$, $B^-_R(\bm{0})$ has Lipschitz boundary with constants $r^*R, E^*$.
Possibly worsening the regularity parameters of $C$, see \eqref{more regularity for C}, we can assume $E_0\geq E^*$ and $r_0\leq r^* R$, so that the boundary of $\Omega$ is of Lipschitz class with constants $r_0$ and $E_0$.

Following similar arguments as those in \cite{Morassi-Rosset} with the simplification of mantaining as integrand function $|\widehat{\nabla}\bm{w}|^2$, we find that
there exist $\varrho_0=\varrho_0(\lambda,\mu,r_0,E_0,R)$, with $0<\varrho_0<1$, and $s=s(\lambda,\mu,E_0)$, $s>1$, such that for all $0<\varrho\leq \varrho_0$ and for all $\overline{\bm{x}}\in\Omega_{s\varrho}$, it holds
	\begin{equation}\label{LPS primitive version}
		\int\limits_{B_{\varrho}(\overline{\bm{x}})}|\widehat{\nabla}\bm{w}|^2\, d\bm{x}\geq c_1\int\limits_{\Omega}|\widehat{\nabla}\bm{w}|^2\, d\bm{x}\left(\frac{c\varrho^{3}\displaystyle\int\limits_{\Omega_{(s+1)\varrho}}|\widehat{\nabla}\bm{w}|^2\, d\bm{x}}{\displaystyle\int\limits_{\Omega}|\widehat{\nabla}\bm{w}|^2\, d\bm{x}}\right)^{\sigma^{-A_1-B_1\log(1/\varrho)}},
	\end{equation}
where $c_1>0$ only depend on $\lambda,\mu$; $\sigma\in(0,1), s>1$ depend on $\lambda,\mu, E_0$ and $c,A_1,B_1>0$ depend on $\lambda,\mu,E_0,r_0,R$.
The main goal is to give a lower estimate of the ratio 
	\begin{equation*}
		\frac{\displaystyle\int\limits_{\Omega_{(s+1)\varrho}}|\widehat{\nabla}\bm{w}|^2\, d\bm{x}}{\displaystyle\int\limits_{\Omega}|\widehat{\nabla}\bm{w}|^2\, d\bm{x}}
	\end{equation*}		
to get the assertion of the theorem.
To this end, we first notice that
\begin{equation}\label{1-ratio I_1 and I_2}
		\frac{\displaystyle\int\limits_{\Omega_{(s+1)\varrho}}|\widehat{\nabla}\bm{w}|^2\, d\bm{x}}{\displaystyle\int\limits_{\Omega}|\widehat{\nabla}\bm{w}|^2\, d\bm{x}}=1-\frac{\displaystyle\int\limits_{\Omega\setminus\Omega_{(s+1)\varrho}}|\widehat{\nabla}\bm{w}|^2\, d\bm{x}}{\displaystyle\int\limits_{\Omega}|\widehat{\nabla}\bm{w}|^2\, d\bm{x}}:=1-\frac{\mathcal{I}_1}{\mathcal{I}_2}.
	\end{equation}
Let us consider the integral $\mathcal{I}_2$. 
Since $\bm{w}=\bm{u}-\overline{\bm{u}}$, see \eqref{definition of the function w}, we use the integral representation formula $\eqref{representation formula}$ for the function $\bm{u}$ and the explicit expression of $\overline{\bm{u}}$ in \eqref{definition of ubar}.
In detail, we consider
	\begin{equation}\label{spherical cap S}
		S=\bigg\{\bm{x}\in B^-_R(\bm{0})\, :\, x_3\leq -\frac{3}{4}R\bigg\}.
	\end{equation} 
see Figure \ref{fig. spherical sector}.
\begin{figure}[h]
\centering
\includegraphics[scale=0.6]{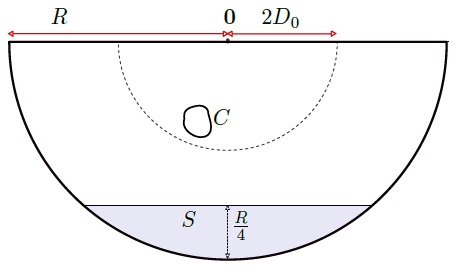}
\caption{The region $S$.}\label{fig. spherical sector}
\end{figure} 	
By a simple calculation we have $|S|=(7/128)\pi R^3$.
 
If $\bm{x}\in\partial C$, $\bm{y}\in S$, then
by \eqref{representation formula} and Theorem \ref{thm:fundsol}, it is easy to see that
	\begin{equation*}
		|\widehat{\nabla}\bm{u}(\bm{y})|\leq |\nabla\bm{u}(\bm{y})|\leq \frac{c p}{R^2},\quad \forall \bm{y}\in S,
	\end{equation*}	
where $c=c(\lambda,\mu,r_0,E_0,D_0)$, hence
	\begin{equation*}
		|\widehat{\nabla}\bm{w}(\bm{y})| \geq |\widehat{\nabla}\overline{\bm{u}}(\bm{y})|-|\widehat{\nabla}\bm{u}(\bm{y})|\geq \frac{p}{3\lambda+2\mu}-\frac{c p}{R^2}\geq \frac{p}{2(3\lambda+2\mu)},
	\end{equation*}
where the last inequality holds choosing $R=\max\{3D_0, (2c^{-1}(3\lambda+2\mu))^{1/2}\}$.
In this way, we have that
	\begin{equation}\label{mathcal I_2}
		\mathcal{I}_2=\int\limits_{\Omega}|\widehat{\nabla}\bm{w}|^2\, d\bm{x}\geq \int\limits_{S}|\widehat{\nabla}\bm{w}|^2\, d\bm{x}\geq c p^2 R^3,
	\end{equation}
where $c=c(\lambda,\mu)$.
	
Now, we estimate the integral $\mathcal{I}_1$ using the regularity result of the Proposition \ref{th regularity estimate}. 
First, we split the integral domain as
\begin{equation*}
\Omega\setminus \Omega_{(s+1)p}=F_1\cup F_2,
\end{equation*}
where
\begin{equation*}
F_1 = \{\bm{x}\in \Omega \, :\, d(\bm{x},C))\leq (s+1)\varrho\},\qquad F_2= \{\bm{x}\in \Omega \, :\, d(\bm{x},\partial B^-_R(\bm{0}))\leq (s+1)\varrho\},
\end{equation*}
see  Figure \ref{fig. bordino}.
\begin{figure}[h!]
\centering
\includegraphics[scale=0.6]{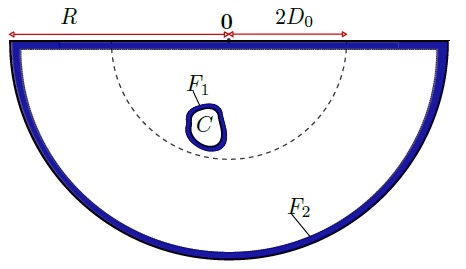}
\caption{The region $\Omega\setminus \Omega_{(s+1)\varrho}$.}\label{fig. bordino}
\end{figure}
\\From \eqref{more regularity for C}, we notice that 
	\begin{equation}\label{estimate of volume of F1 and F2}
		|F_1\cup F_2|\leq c(r_0,E_0,D_0) \varrho R^2. 
	\end{equation}
Choosing in \eqref{Omega domain inclusions} $\alpha=R/D_0$, where $R=\max\{3D_0, (2c^{-1}(3\lambda+2\mu))^{1/2}\}$, and $\beta=2\alpha$, then 
	\begin{equation*}
		\overline{B^-_{R}(\bm{0})}\subset\subset \overline{Q}\subset\subset \overline{B^-_{2R}(\bm{0})},
	\end{equation*} 
hence we apply the regularity estimate \eqref{eq regularity estimate} for the two regions $F_1$ and $F_2$. 
Now, we have that
	\begin{equation*}
		\|\widehat{\nabla}\bm{w}\|_{L^{\infty}(\Omega\setminus \Omega_{(s+1)\varrho})}\leq \big(\|\widehat{\nabla}\bm{u}\|_{L^{\infty}(\Omega\setminus \Omega_{(s+1)\varrho})}+\|\widehat{\nabla}\overline{\bm{u}}\|_{L^{\infty}(\Omega\setminus \Omega_{(s+1)\varrho})})\leq c p,
	\end{equation*}
where $c=c(\lambda,\mu,r_0,E_0, D_0)$.
Therefore, from \eqref{estimate of volume of F1 and F2}, we find 
	\begin{equation}\label{mathcal I_1}
		\mathcal{I}_1=\int\limits_{\Omega\setminus\Omega_{(s+1)\varrho}}|\widehat{\nabla}\bm{w}|^2\, d\bm{x}\leq cp^2\, |\Omega\setminus\Omega_{(s+1)\varrho}|\leq cp^2 \varrho R^2. 
	\end{equation} 	
Putting together inequalities \eqref{mathcal I_2} and \eqref{mathcal I_1}, there exists $\varrho^*=\varrho^*(\lambda,\mu,r_0,E_0,D_0)>0$, such that for any $\varrho\leq \varrho^*$, we have
	\begin{equation*}
		\frac{\mathcal{I}_1}{\mathcal{I}_2}=\frac{\displaystyle\int\limits_{\Omega\setminus\Omega_{(s+1)\varrho}}|\widehat{\nabla}\bm{w}|^2\, d\bm{x}}{\displaystyle\int\limits_{\Omega}|\widehat{\nabla}\bm{w}|^2\, d\bm{x}}\leq \frac{1}{2}.
	\end{equation*}	
Going back to \eqref{LPS primitive version} and \eqref{1-ratio I_1 and I_2} we have
	\begin{equation*}
		\int\limits_{B_{\varrho}(\overline{\bm{x}})}|\widehat{\nabla}\bm{w}|^2\, d\bm{x}\geq \bigg(c\varrho^3\bigg)^{\sigma^{-A_1-B_1\log(1/\varrho)}}\int\limits_{\Omega}|\widehat{\nabla}\bm{w}|^2\, d\bm{x}, 
	\end{equation*} 
where $c,A_1,B_1$ depend on $\lambda,\mu,E_0,r_0,D_0$, for all $\varrho\leq \varrho^*$.
To conclude, we take $\varrho\leq c$, hence, for every $\varrho\leq \min(c,\varrho^*)$ and noticing that $\log \varrho \geq -1/\varrho$ for $0<\varrho<1$, we get the assertion choosing
	\begin{equation*}
		a=6e^{A_1|\log\sigma|}\qquad\quad \textrm{and}\qquad\quad b=B_1|\log\sigma|+1.
	\end{equation*}  
\end{proof}

We omit the proof of the following two propositions, since they can be obtained using the same strategy adopted in proving Propositions 3.5 and 3.6 in \cite{Morassi-Rosset}.

\begin{proposition}[Stability Estimates of Continuation from Cauchy Data]\label{prop: SECCD}
Under the assumption \eqref{bounds_lame} let $C_1$ and $C_2$ be two domains satisfying \eqref{a priori information on C}, \eqref{first a priori information}, \eqref{second a priori information} and \eqref{more regularity for C}. Moreover, let $\bm{w}_i$, for $i=1,2$, be the solution to \eqref{pb for w} with $C=C_i$. Then, for $\varepsilon<e^{-1}p$, we have
	\begin{equation}\label{stability estimates from cauchy data}
		\begin{aligned}
			&\int\limits_{C_2\setminus \overline{C}_1}|\widehat{\nabla}\bm{w}_1|^2\, d\bm{x} \leq c p^2 \bigg(\log\Big|\log\frac{\varepsilon}{p}\Big|\bigg)^{-1/6},\\
			&\int\limits_{C_1\setminus \overline{C}_2}|\widehat{\nabla}\bm{w}_2|^2\, d\bm{x} \leq c p^2 \bigg(\log\Big|\log\frac{\varepsilon}{p}\Big|\bigg)^{-1/6},
		\end{aligned}
	\end{equation}
where the constant $c=c(\lambda,\mu, r_0,E_0,D_0,s_0)$.
\end{proposition}

The stability estimates in \eqref{stability estimates from cauchy data} can be improved
when $\partial G$ is of Lipschitz class, where $G$ is defined by \eqref{G unbounded component}, as stated in the proposition below.

\begin{proposition}[Improved Stability Estimates of Continuation from Cauchy Data]\label{improved stability estimates of continuation from cauchy data}
Under the assumption \eqref{bounds_lame} let $C_1$ and $C_2$ be two domains satisfying \eqref{a priori information on C}, \eqref{first a priori information}, \eqref{second a priori information} and \eqref{more regularity for C}. In addition, let us assume that there exist $L>0$ and $\widetilde{r}_0$, with $0<\widetilde{r}_0\leq r_0$, such that $\partial G$ is of Lipschitz class with constants $\widetilde{r}_0, L$. Then, we have
	\begin{equation}\label{improved stability estimates from cauchy data}
		\begin{aligned}
			&\int\limits_{C_2\setminus \overline{C}_1}|\widehat{\nabla}\bm{w}_1|^2\, d\bm{x} \leq c p^2 \Big|\log\frac{\varepsilon}{p}\Big|^{-\gamma},\\
			&\int\limits_{C_1\setminus \overline{C}_2}|\widehat{\nabla}\bm{w}_2|^2\, d\bm{x} \leq c p^2 \Big|\log\frac{\varepsilon}{p}\Big|^{-\gamma},
		\end{aligned}
	\end{equation}
where $c,\gamma>0$ depend on $\lambda,\mu, r_0,E_0,D_0,s_0,L,\widetilde{r}_0$.
\end{proposition}

Now, we have all the preliminary results to prove the stability theorem.
\begin{theorem}[Stability Estimate]\label{th: stability estimate}
Under the assumption \eqref{bounds_lame} let $C_1$ and $C_2$ be two domains satisfying \eqref{a priori information on C}, \eqref{first a priori information}, \eqref{second a priori information} and \eqref{more regularity for C}. Moreover, let $\bm{u}_i$, for $i=1,2$, be the solution to \eqref{direct pb} with $C=C_i$. If, given $\varepsilon>0$, we have
	\begin{equation}
		\|\bm{u}_1-\bm{u}_2\|_{L^2(B'_{s_0}(\bm{0}))}\leq \varepsilon,
	\end{equation}
then it holds
	\begin{equation}
		d_{\mathcal H}(\partial C_1,\partial C_2)\leq c \bigg(\log\Big|\log \frac{\varepsilon}{p}\Big|\bigg)^{-\eta},
	\end{equation}
for every $\varepsilon<e^{-1}p$, where the constants $c$ and $\eta$, with $0<\eta\leq 1$, depend on $\lambda,\mu,r_0,E_0,D_0$ and $s_0$.		
\end{theorem}
\begin{proof}
Since \eqref{w_1-w_2=u_1-u_2} holds, we prove the assertion using the function $\bm{w}_i$, for $i=1,2$. 
In this way we can apply the same proof strategy contained in \cite{Morassi-Rosset1}. In the sequel we simply denote with $d_{\mathcal{H}}$ the Hausdorff distance $d_{\mathcal{H}}(\partial C_1,\partial C_2)$.
 
Let us prove that if $\eta>0$ is such that
	\begin{equation}\label{integral grad w less of eta}
		\int\limits_{C_2\setminus \overline{C}_1}|\widehat{\nabla}\bm{w}_1|^2\, d\bm{x}\leq \eta,\quad\qquad
		\int\limits_{C_1\setminus \overline{C}_2}|\widehat{\nabla}\bm{w}_2|^2\, d\bm{x}\leq \eta,
	\end{equation}
then we have
	\begin{equation}\label{distance d in the stability estimate}
		d_{\mathcal H} \leq c \bigg(\log\frac{cp^2}{\eta}\bigg)^{-1/b},
	\end{equation} 
where $b,c$ depend on $\lambda,\mu,r_0,E_0,D_0$.\\
We may assume, with no loss of generality, that there exists $\bm{x}_0\in\partial C_1$ such that
	\begin{equation*}
		\textrm{dist}(\bm{x}_0,\partial C_2)=d_{\mathcal H}.
	\end{equation*}
In this setting, we have to distinguish two cases:
	\begin{enumerate}[(i)]
		\item $B_{d_{\mathcal H}}(\bm{x}_0)\subset C_2$;\label{first inclusion}
		\item $B_{d_{\mathcal H}}(\bm{x}_0)\cap C_2=\emptyset$. \label{intersection}
	\end{enumerate}
Let us consider case \eqref{first inclusion}.
By the regularity assumption made on $\partial C_1$, see \eqref{more regularity for C}, there exists $\bm{x}_1\in C_2\setminus \overline{C}_1$ such that 
	\begin{equation*}
		B_{tdd_{\mathcal H}}(\bm{x}_1)\subset (C_2\setminus \overline{C}_1),\, \textrm{with}\, \,\, t=\frac{1}{1+\sqrt{1+E^2_0}}.
	\end{equation*}
By the first inequality in \eqref{integral grad w less of eta}, taking $\varrho=tdd_{\mathcal H}/s$
in Proposition \ref{LPS}, we have
	\begin{equation}\label{inequality for eta}
		\eta\geq \int\limits_{C_2\setminus \overline{C}_1}|\widehat{\nabla}\bm{w}_1|^2\, d\bm{x}\geq \int\limits_{B_{\varrho}(\bm{x}_1)}|\widehat{\nabla}\bm{w}_1|^2\, d\bm{x}\geq \frac{c}{e^{a\varrho^{-b}}}\int\limits_{B^-_R\setminus \overline{C}_1}|\widehat{\nabla}\bm{w}_1|^2\, d\bm{x},
	\end{equation} 
where we recall that $R=R(\lambda,\mu,r_0,E_0,D_0)$. By \eqref{mathcal I_2}, we find that
	\begin{equation*}
		\int\limits_{B^-_R\setminus \overline{C}_1}|\widehat{\nabla}\bm{w}_1|^2\, d\bm{x}\geq c p^2,
	\end{equation*} 
so that, going back to \eqref{inequality for eta}, we have
	\begin{equation}
		\eta\geq \frac{c p^2}{\displaystyle e^{a\varrho^{-b}}}=\frac{c p^2}{\displaystyle e^{a(td_{\mathcal H}/s)^{-b}}}.
	\end{equation}	 
From this inequality it is straightforward to find \eqref{distance d in the stability estimate}.

Case \eqref{intersection} can be proved in a similar way by substituting $\bm{w}_1$ with $\bm{w}_2$ in the previous calculations and employing the second inequality in \eqref{integral grad w less of eta}.

Now, applying \eqref{stability estimates from cauchy data}, that is taking
	\begin{equation*}
		\eta=c p^2 \bigg(\log\Big|\log\frac{\varepsilon}{p}\Big|\bigg)^{-1/6},
	\end{equation*}
we obtain from \eqref{distance d in the stability estimate} that
	\begin{equation} \label{3log}
		d_{\mathcal H}\leq c\bigg(\log\log\Big|\log\frac{\varepsilon}{p}\Big|\bigg)^{-1/b},
	\end{equation}
where we require $\varepsilon< e^{-e}p$ to have a positive quantity in right side of the previous inequality; the positive constants $b,c$ depend on $\lambda,\mu,r_0,E_0,s_0$ and $D_0$.

Next, to improve the modulus of continuity of this estimate we recall a geometrical result, first introduced and proved in \cite{Alessandrini-Beretta-Rosset-Vessella}, ensuring that there exists $d_0>0$, $d_0=d_0(r_0,E_0)$ such that if $d_{\mathcal H}(\partial C_1,\partial C_2)\leq d_0$, then the boundary of $G$ is of Lipschitz class with constants $\widetilde{r}_0$, $L$, only depending on $r_0$ and $E_0$. By \eqref{3log}, there exists $\varepsilon_0>0$ only depending on $\lambda,\mu,r_0,E_0,s_0$ and $D_0$ such that if $\varepsilon\leq \varepsilon_0$ then $d_{\mathcal H}\leq d_0$. 				
In this way $G$ satisfies the hypotheses of Proposition \ref{improved stability estimates of continuation from cauchy data} hence the assertion follows. 
\end{proof}

\appendix
\appendixtitleon
\begin{appendices}
\section{Neumann Function for the Lam\'e operator in the Half-Space}
This appendix is devoted to the explicit expression of the Neumann function for the Lam\'e operator in the half-space presented in \cite{Mindlin36,Mindlin54}. Before doing that we recall the fundamental solution $\mathbf{\Gamma}$ of the Lam\'e operator, that is the so called Kelvin-Somigliana matrix. $\mathbf{\Gamma}$ is the solution to the equation
\begin{equation*}
	\textrm{div}(\mathbb{C}\widehat{\nabla}\bm{\Gamma})=\delta_{\bm{0}}\mathbf{I},\qquad
	\bm{x}\in\mathbb{R}^3\setminus\{\bm{0}\}, 
\end{equation*}
where $\delta_{\bm{0}}$ is the Dirac function centred at $\bm{0}$ and $\mathbf{I}$ is the identity matrix.
Setting $C_{\mu,\nu}:={1}/\{16\pi\mu(1-\nu)\}$, where $\nu$ is the Poisson ratio $\nu=\lambda/(2(\lambda+\mu))$, the explicit expression of $\mathbf{\Gamma}=(\Gamma_{ij})$ is
\begin{equation}\label{Gamma}
	\Gamma_{ij}(\bm{x})=-C_{\mu,\nu}\biggl\{\frac{(3-4\nu)\delta_{ij}}{|\bm{x}|}+\frac{x_i x_j}{|\bm{x}|^3}\biggr\},
	\qquad i,j=1,2,3,
\end{equation}
where $\delta_{ij}$ is the Kronecker symbol.

Given $\bm{y}=(y_1,y_2,y_3)$, we set $\widetilde{\bm{y}}=(y_1,y_2,-y_3)$. Now, we have  
\begin{theorem}[\cite{Aspri-Beretta-Mascia}]\label{thm:fundsol}
The Neumann function $\mathbf{N}$ of problem \eqref{Neumann function} can be decomposed as
\begin{equation*}
	\mathbf{N}(\bm{x},\bm{y})=\mathbf{\Gamma}(\bm{x}-\bm{y})+\mathbf{R}^1(\bm{x}-\widetilde{\bm{y}})
		+y_3\mathbf{R}^2(\bm{x}-\widetilde{\bm{y}})+y_3^2\,\mathbf{R}^3(\bm{x}-\widetilde{\bm{y}}),
\end{equation*}
where $\mathbf{\Gamma}$ is the Kelvin matrix, see \eqref{Gamma},
and $\mathbf{R}^k$, $k=1,2,3$, have components $R^k_{ij}$ given by
\begin{equation*}
	\begin{aligned}
	R^1_{ij}(\bm{\eta})&:=C_{\mu,\nu}\bigl\{-(\tilde f+c_\nu\tilde g)\delta_{ij}-(3-4\nu)\eta_i\eta_j\tilde f^3\\
	&\hskip2.75cm +c_\nu\bigl[\delta_{i3}\eta_j-\delta_{j3}(1-\delta_{i3})\eta_i\bigr]\tilde f\tilde g
		+c_\nu(1-\delta_{i3})(1-\delta_{j3})\eta_i \eta_j\tilde f\tilde g^2\bigr\}\\
	R^2_{ij}(\bm{\eta})&:=2C_{\mu,\nu}\bigl\{(3-4\nu)\bigl[\delta_{i3}(1-\delta_{j3})\eta_j+\delta_{j3}(1-\delta_{i3})\eta_i\bigr]\tilde f^3
		-(1-2\delta_{3j})\delta_{ij}\eta_3\tilde f^3\\
	&\hskip8.75cm +3(1-2\delta_{3j})\eta_i\eta_j\eta_3\tilde f^5\bigr\}\\
	R^3_{ij}(\bm{\eta})&:=2C_{\mu,\nu}(1-2\delta_{j3})\bigl\{\delta_{ij} \tilde f^3-3\eta_i\eta_j\tilde f^5\bigr\}.
	\end{aligned}
\end{equation*}
for $i,j=1,2,3$, where $c_\nu:=4(1-\nu)(1-2\nu)$ and
\begin{equation*}
	\tilde f(\bm{\eta}):=\frac{1}{|\bm{\eta}|},\qquad
	\tilde g(\bm{\eta}):=\frac{1}{|\bm{\eta}|-\eta_3}.
\end{equation*}
\end{theorem}
\end{appendices}

\section*{Acknowledgements}
The authors thank Prof. Cherif Amrouche for his kindness in suggesting and providing some useful papers and for his enlightening advice.
Andrea Aspri and Elena Beretta thank the New York University in Abu Dhabi (EAU) for its kind hospitality that permitted a further 
development of the present research. Andrea Aspri thanks \"{O}AW (Austrian Academy of Sciences) and RICAM for giving him the possibility to finish this paper.
Edi Rosset is supported by FRA2016 "Problemi inversi, dalla stabilit\`a alla ricostruzione", Universit\`a degli Studi di Trieste and by Progetto GNAMPA 2017 "Analisi di problemi inversi: stabilit\`a e ricostruzione", Istituto Nazionale di Alta Matematica (INdAM).

\end{document}